\documentclass[12pt]{amsart}
\pdfoutput=1
\usepackage{graphicx}
\usepackage{amssymb}
\usepackage{amsmath}
\usepackage{amsthm}
\usepackage{amscd}
\usepackage{epsfig}

\newtheorem{theorem}{Theorem}[section]

\newtheorem{lemma}[theorem]{Lemma}

\newtheorem{proposition}[theorem]{Proposition}

\newtheorem{question}[theorem]{Question}
\newtheorem{remark}[theorem]{Remark}
\theoremstyle{definition}

\numberwithin{equation}{section}

\newcommand{\Z}{\mathbb Z}
\newcommand{\N}{\mathbb N}

\begin{document}

\title{Minimal subshifts of arbitrary mean topological dimension}

\author [Dou Dou]{Dou Dou}

\address{Department of Mathematics, Nanjing University,
Nanjing, Jiangsu, 210093, P.R. China} \email{doumath@163.com}

\thanks{}


\begin{abstract}
Let $G$ be a countable infinite amenable group and $P$ be a polyhedron.
We give a construction of minimal subshifts of $P^G$ with arbitrarily mean topological dimension less than $\dim P$.
\end{abstract}

\maketitle

\section{Introduction}
In 1999, Gromov \cite{G} introduced a topological invariant of dynamical systems which is called mean dimension or mean topological dimension.
A series of research \cite{Gu,Gu2,Gu3,GLT,GM,L,LT,LW} show that mean dimension is naturally connected to the embedding problems via $(([0,1]^d)^{\Z},\text{shift})$ and $(([0,1]^d)^{\Z^k},\text{shift})$.
For example, using mean dimension, Lindenstrauss and Weiss \cite{LW} constructed a minimal system with mean dimension greater than $1$ and gave a negative answer
to the long-standing open question that whether every minimal system can be embedded into $([0,1]^{\Z},\text{shift})$ (see Auslander \cite{A}).
They also studied the basic theory of mean dimension and explained that their results can be generalized to actions of countable infinite amenable
groups.

In the case of actions of countable infinite amenable groups, it is natural to ask the following question:

What are the possible values for the mean topological dimension? How about the minimal case?

Lindenstrauss and Weiss \cite{LW}
assert that (without detail proof) for the $\Z$-action case, the values for the mean topological dimension of minimal systems can take over $[0,+\infty]$.
Although people believe that using Ornstein and Weiss's quasi-tiling theory \cite{OW}, similar construction and result can be extend to actions of countable infinite amenable groups,
to our knowledge, there are only partial answers to these questions.

When the countable infinite amenable group, denoted by $G$, has subgroups of arbitrarily large finite index, given a polyhedron $P$ and
a nonnegative real number $\rho$ which is no more than the topological dimension of $P$, Coornaert and Krieger \cite{CK} (see also \cite{C}) constructed a closed
subshift $X\subset P^G$ having mean topological dimension $\rho$. Hence in this situation, the mean topological dimension of $G$-systems can take all values in $[0,+\infty]$.
We remark here that this case includes all residually finite countable infinite amenable groups. When the countable infinite amenable group $G$ is residually finite, for any $\varepsilon>0$, Krieger \cite{K1}
constructed minimal systems with mean topological dimension $\varepsilon$-close to the topological dimension of $P$.
A later work by Krieger \cite{K2} improved $G$ to general countable infinite amenable groups.

In this paper, we will provide detail constructions of minimal subshifts of $P^G$ with arbitrary mean topological dimension less than $\dim P$ without any other restrictions
to the countable infinite amenable group $G$. Hence for any countable infinite amenable group $G$, the mean topological dimension of minimal $G$-systems takes all values in $[0,+\infty]$.

We will apply the classical construction of Toeplitz minimal flows for $\Z$-actions (see \cite{WS}). The key idea of that is to find some suitable ``periodic skeleton".
Since we add no restrictions to the amenable group $G$, the proper periodic structures in $G$ will be replaced by some ``almost periodic'' structures. This kind of ``almost periodic'' structures can be determined from the quasi-tiling techniques. To handle these ``almost periodic'' structures clearly, we will employ a recent result due to Downarowicz etc. \cite{DHZ}
which says that every countable infinite amenable group admits a congruent sequence of finite tilings. By showing the existence of some specific sequence of finite tilings, which will be called irreducible tilings in section 3, we can obtain the required ``almost periodic'' structures.

The paper is organized as follows. In section 2, we introduce some basic knowledge on the amenable group, subshift and mean topological dimension. Section 3 is devoted to discussing
the irreducible tilings of amenable groups and giving preparations for our constructions. In section 4, we will show our detail constructions of minimal subshifts with arbitrary mean topological dimension.

\section{Preliminaries}
\subsection{Amenable groups}
Recall that a group $G$ is {\it amenable} if there exists a sequence of nonempty
finite subsets $\{F_n\}$ of $G$ which are asymptotically invariant, i.e.,
$$\lim_{n\rightarrow+\infty}\frac{|F_n\vartriangle gF_n|}{|F_n|}=0, \text{ for all } g\in G.$$
Such sequences are called F{\o}lner sequences. For the detail of amenable group actions, one may refer to Ornstein and Weiss \cite{OW}.

Let $F(G)$ denote the collection of nonempty finite subsets of $G$ and $A,K\subset F(G)$. The {\it $K$-boundary} of $A$ is defined by
$$B(A,K)=\{g\in G: Kg\cap A\neq\emptyset \text { and } Kg\cap(G\setminus A)\neq\emptyset\}.$$
For $\delta>0$, the set $A$ is said to be {\it $(K, \delta)$-invariant} if
\begin{align*}
  \frac{|B(A,K)|}{|A|}<\delta.
\end{align*}
\subsection{$G$-systems and subshifts}
A $G$-system is a couple $(X,G)$ where $G$ is a group acting continuously on a compact metrizable space $X$. A $G$-system $(X,G)$ is said to be
{\it minimal} if for any $x\in X$, its orbit $\{gx: g\in G\}$ is dense in $X$.

Recall that a subset $S\subset G$ is called {\it syndetic} if there exists a finite subset $F\subset G$ such that $G=FS$.
A point $x\in X$ is said to be a {\it minimal point} ( or an {\it almost periodic point}) if for any neighbourhood $V$ of $x$, the set $\{g\in G: gx\in V\}$ is syndetic.
It is known that $(X,G)$ is minimal if and only if $X$ is the orbit closure of a minimal point.

Let $K$ be a compact metrizable space. The full $G$-shift with space of symbols $K$ is
the $G$-system given by the left action of $G$ on the product space $K^G=\{(x_g)_{g\in G}: x_g\in K\}$
defined by
$$g'(x_g)_{g\in G}=(x_{gg'})_{g\in G},$$
for all $g'\in G$ and $(x_g)_{g\in G}\in K^G$. A closed $G$-invariant subset
of $K^G$ is called a {\it subshift}. For a point $x=(x_g)_{g\in G}\in K^G$ and $F\subset G$, we denote by $x|_F=(x_g)_{g\in F}$,
the restriction of $x$ on $F$.

Let $F\subset G$ and $c\in G$. For $u=(u_g)_{g\in F}\in K^F$ and $v=(v_g)_{g\in Fc}\in K^{Fc}$, if $u=cv$, i.e. $u_g=v_{gc}$ for all $g\in F$, we sometime denote it by $u=v$ when there were no ambiguity. We also call an element in $K^F$ or $K^G$ a $K$-word or a word.

\subsection{Mean topological dimension}
Let $X$ be a compact metrizable  space and $\alpha=\{U_1,U_2,\ldots, U_k\}$ be a finite open cover of $X$. The {\it order}
of $\alpha$ is defined by
$${\rm ord}(\alpha)=\max_{x\in X}\sum_{i=1}^k1_{U_i}(x)-1.$$
Denote by $${\rm D}(\alpha)=\min_{\beta}{\rm ord}(\beta),$$
where $\beta$ is taken over all finite open cover of $X$ with $\beta\succ\alpha$.

The {\it topological dimension} of $X$ is then defined by
$$\dim X=\sup_{\alpha} {\rm D}(\alpha),$$
where $\alpha$ runs over all finite open cover of $X$.

Let $(X,G)$ be a $G$-system, where $G$ is an amenable group. For $F\in F(G)$ and a finite open cover $\alpha$ of $X$, denote by $\alpha_F=\bigvee_{g\in F}g^{-1}\alpha$.
Then we can define
$${\rm D}(\alpha, G)=\lim_{n\rightarrow \infty}\frac{{\rm D}(\alpha_{F_n})}{|F_n|},$$
where $\{F_n\}$ is a F{\o}lner sequence of $G$. It is known that this limit exists and is independent on the choice of the F{\o}lner sequence.
The {\it mean topological dimension} ${\rm mdim}(X,G)$ of $(X,G)$ is defined by
$${\rm mdim}(X,G)=\sup_{\alpha} {\rm D}(\alpha,G),$$
where $\alpha$ runs over all finite open cover of $X$.

For $J\subset G$, the density $\delta(J)\subset[0,1]$ of $J$ in $G$ is defined by
$$\delta(J)=\sup_{\{F_n\}}\limsup_{n\rightarrow\infty}\frac{|J\cap F_n|}{|F_n|},$$
where $\{F_n\}$ is taken over all F{\o}lner sequences of $G$.

The following is Proposition 2.8 of \cite{K2} which will be used for estimating the lower bound of the mean topological dimension.
We note here that this proposition is originally due to Lindenstraus and Weiss (\cite{LW}, Proposition 3.3) for $\Z$-actions.
\begin{proposition}\label{lowerbound}
  Let $G$ be a countable amenable group. Let $X\subset P^G$ be a closed subshift, where $P$ is a polyhedron.
  For a subset $E$ of $G$, we denote by $\pi_E$ the projection map $P^G\rightarrow P^E$. Suppose that there exist
  $x'=(x'_g)_{g\in G}\in X$ and a subset $J\subset G$ satisfying the following condition
  $$\pi_{G\setminus J}(x)=\pi_{G\setminus J}(x')\Rightarrow x\in X,$$
  for all $x\in P^G$.

  Then we have
  $${\rm mdim}(X,G)\ge \delta(J)\dim(P).$$
\end{proposition}

To estimate the upper bound of the mean topological dimension, we will employ Gromov's ``Pro-Mean Inequality" (\cite{G}, Proposition 1.9.1).
Here we use the following statement by Coornaert and Krieger (\cite{CK}, Proposition 4.1).

\begin{proposition}\label{upperbound}
  Let $K$ be a compact metrizable space of finite topological dimension. Let $X\subset K^G$ be a closed subshift. Then one has
  $${\rm mdim}(X,G)\le \liminf_{n\rightarrow\infty}\frac{\dim(\pi_{F_n}(X))}{|F_n|},$$
  for every F{\o}lner sequence $\{F_n\}$ of $G$, where for a subset $E$ of $G$, $\pi_E$ is the projection map $K^G\rightarrow K^E$.
\end{proposition}

\section{Irreducible tilings of amenable group}
Let $G$ be a group and $F(G)$ be the collection of nonempty finite subsets of $G$.
We call $\mathcal{T}\subset F(G)$ a {\it tiling} if $\mathcal{T}$ forms a partition of $G$. An element in a tiling $\mathcal{T}$ is called
a $\mathcal{T}$-tile or tile.
A tiling $\mathcal{T}$ is said to be finite if there exists a finite collection $\mathcal{S}=\mathcal{S}(\mathcal{T})=\{S_1, S_2,\cdots,S_k\}$ of $F(G)$,
which is called the shapes of $\mathcal{T}$, such that each element in $\mathcal{T}$ is a translation of some set in $\mathcal{S}$.
For convenience, we always assume that the shapes $\mathcal{S}$ has minimal cardinality, i.e. any set in $\mathcal{S}$ cannot
be a translation of others. Moreover, by some suitable translation, we may require that each set in $\mathcal{S}$ contains $e_G$.

Let $S$ be a shape of a finite tiling $\mathcal{T}$, the center of shape $S$ is the set $C(S)=\{c\in G: Sc\in \mathcal{T}\}$. For convenience,
we need $C(S)$ to be nonempty for each shape $S$.
We also require the centers $C(S)$'s satisfy that $Sc$'s are disjoint for $c\in C(S)$ and $S\in\mathcal{S}$.

Let $G$ be an infinite amenable group. A finite tiling $\mathcal{T}$ of $G$ is said to be {\it irreducible} if there exist $T\in F(G)$ and $\epsilon>0$ such that for any $(T,\epsilon)$-invariant finite
subset $F$ in $G$ and any shape $S$ of $\mathcal{T}$, there is a translation of $S$ in $\mathcal{T}$ which is contained in $F$.

For a tiling $\mathcal{T}$ with shapes $\mathcal{S}$, we can define a subshift $X_{\mathcal{T}}$ of $(\mathcal{S}\cup\{0\})^G$ by
$$X_{\mathcal{T}}=\overline{\cup_{g\in G}gx},$$
where $x=(x_g)_{g\in G}$ is defined by
\begin{align*}
  x_g=\begin{cases}
    S, \text{ if }g\in C(S),\\
    0, \text{ otherwise},
  \end{cases}
\end{align*}
i.e., $x$ is a transitive point of the subshift $X_{\mathcal{T}}$.

\begin{lemma}\label{irreducible}
Let $\mathcal{T}$ be a finite tiling of an infinite amenable group $G$ with shapes $\mathcal{S}$. Then the following are equivalent:
\begin{enumerate}
  \item $\mathcal{T}$ is irreducible;
  \item for any $y\in X_{\mathcal{T}}$ and any shape $S\in\mathcal{S}$, there exists $g\in G$ with $y_g=S$;
  \item for any shape $S\in\mathcal{S}$, the center of $S$ is syndetic.
\end{enumerate}
\end{lemma}
\begin{proof}
$(1)\Rightarrow(2)$. Let $T\in F(G)$ and $\epsilon>0$ such that whenever $F\in F(G)$ is $(T,\epsilon)$-invariant
and $S\in\mathcal{S}$, we can find a tile with shape $S$ contained in $F$. Let $F$ be any $(T,\epsilon)$-invariant finite subset of $G$.
Suppose $y=\lim_{n\rightarrow\infty}g_nx$. Then there exists $n$ such that $y|_F=(g_nx)|_F=x|_{Fg_n}$. Since $Fg_n$ is also $(T,\epsilon)$-invariant,
we can find for each $S\in \mathcal{S}$ an element $c\in C(S)$ such that $Sc\subset Fg_n$. In particular, $c\in Fg_n$. Hence $y_{cg_n^{-1}}=x_c=S$.

$(2)\Rightarrow(3)$. Let $\{F_n\}$ be a F{\o}lner sequence of $G$ satisfying that $\{e_{G}\}\subset F_1\subset F_2\subset \cdots$ and $\cup_{n=1}^\infty F_n=G$.
Let $F_n'=F_nF_n^{-1}$ for each $n$. If for some $S\in \mathcal{S}$, $C(S)$ is not syndetic, then for each $n$ there exists $g_n\in G$ such that $g_n\notin F_n'C(S)$.
Hence $c\notin F_n'g_n$ for any $c\in C(S)$ and then $S$ does not occur in $(g_n^{-1}x)|_{F_n'}$. Let $y$ be a limit point of the sequence $(g_n^{-1}x)_{n\ge 1}$.
It is easy to see that $S$ does not occur in $y$, which contradicts to (2).

$(3)\Rightarrow(1)$. If $C(S)$ is syndetic for $S\in \mathcal{S}$, then there exists $T\in F(G)$ such that $RC(S)=G$, where $R=T\cup \{e_G\}$. Since $e_{G}\in S$,
we have that $RSC(S)=G$. Now let $0<\epsilon<1$ and let $F\in F(G)$ be any $(RSS^{-1}R^{-1},\epsilon)$-invariant subset of $G$. We can find some $g\in F$ such that
$RSS^{-1}R^{-1}g\subset F$. Since $RSC(S)=G$, we may assume $g=rsc$ for some $r\in R, s\in S$ and $c\in C(S)$. Hence
$$Sc\subset RSc\subset RSS^{-1}R^{-1}rsc=RSS^{-1}R^{-1}g\subset F.$$
This implies that $\mathcal{T}$ is irreducible.
\end{proof}

\begin{lemma}
  Let $\mathcal{T}$ be a finite tiling of an infinite amenable group $G$. If $\mathcal{T}$ is irreducible, then for any shape $S\in \mathcal{S}$ and any $n\in\N$,
  there exists $T\in F(G)$ and $\epsilon>0$ such that whenever $F\in F(G)$ is $(T,\epsilon)$-invariant, $F$ contains at least $n$ elements in $\mathcal{T}$ with
  shape $S$.
\end{lemma}
\begin{proof}
  Since $\mathcal{T}$ is irreducible, for any shape $S\in \mathcal{S}$, there exist $T'\in F(G)$ and $\epsilon'>0$ such that any $(T',\epsilon')$-invariant finite subset of $G$ contains an element in $\mathcal{T}$ with shape $S$. Let $A$ be such a $(T',\epsilon')$-invariant set and $Ag_1, Ag_2, \ldots, Ag_n$ be disjoint translations of $A$. Let $T=\cup_{i=1}^nAg_i$ and $0<\epsilon<1$. Then for any $(T,\epsilon)$-invariant finite subset $F$, $F$ contains at least one translation of $T$, say $Tg$. Hence $Ag_ig\subset F$ for $i=1,2,\ldots,n$.
  Since each $Ag_ig$ is $(T',\epsilon')$-invariant, it contains an element in $\mathcal{T}$ with shape $S$. Hence $F$ contains at least $n$ elements in $\mathcal{T}$ with
  shape $S$.
\end{proof}
\begin{lemma}\label{mini}
  Let $\mathcal{T}$ be a finite tiling of an infinite countable amenable group $G$ and $(X_{\mathcal{T}},G)$ be the associated subshift.
  If $(X_{\mathcal{T}},G)$ is minimal, then $\mathcal{T}$ is irreducible.
\end{lemma}
\begin{proof}
  If $(X_{\mathcal{T}},G)$ is minimal, then for each $S\in\mathcal{S}$, $C(S)=\{g\in G: x_g=S\}$ is syndetic. By Lemma \ref{irreducible},
   $\mathcal{T}$ is irreducible.
\end{proof}

Let $\mathcal{T}_1$ and $\mathcal{T}_2$ be two finite tilings of $G$. If every tile of $\mathcal{T}_{2}$
equals a union of tiles of $\mathcal{T}_1$, i.e., as a partition of $G$, $\mathcal{T}_1$ is a refinement of $\mathcal{T}_{2}$ (denoted by $\mathcal{T}_{2}\preccurlyeq\mathcal{T}_1$), then we say $\mathcal{T}_2$ is {\it congruent} with $\mathcal{T}_1$.
If in addition, each shape of ${\mathcal{T}}_{2}$ is partitioned by shapes of ${\mathcal{T}}_{1}$ through a unique way (this kind of partition is called ``the master partition'' by Downarowicz etc \cite{DHZ}), then we say $\mathcal{T}_2$ is {\it primely congruent} with $\mathcal{T}_1$. A sequence of tilings $(\mathcal{T}_k)_{k\ge 1}$ is said to be (primely) congruent if for each $k\ge 1$, $\mathcal{T}_{k+1}$ is (primely) congruent with $\mathcal{T}_{k}$.

The following proposition appears in the proof of Theorem 5.2 of \cite{DHZ}.
\begin{proposition}\label{factor}
  Let $\mathcal{T}_1$ and $\mathcal{T}_2$ be two finite tilings of an infinite countable amenable group $G$ and $\mathcal{T}_2$ is primely congruent with $\mathcal{T}_1$.
  Then $(X_{\mathcal{T}_1},G)$ is a topological factor of $(X_{\mathcal{T}_2},G)$, where $(X_{\mathcal{T}_1},G)$ and $(X_{\mathcal{T}_2},G)$ are the subshifts associated with
  $\mathcal{T}_1$ and $\mathcal{T}_2$ respectively.
\end{proposition}
\begin{proof}
  For $j=1,2$, let $\mathcal{S}_j=\mathcal{S}(\mathcal{T}_j)$ be the shape set of $\mathcal{T}_j$. Let $x^{(j)}=(x^{(j)}_g)_{g\in G}$
  be the associated transitive point of $(X_{\mathcal{T}_j},G)$. We will define a map $\pi$ from $(X_{\mathcal{T}_2},G)$ to $(X_{\mathcal{T}_1},G)$ by the following way. For $x^{(2)}$, we let $\pi(x^{(2)})=x^{(1)}$. Let $y$ be a point in $X_{\mathcal{T}_2}$, assume $y=\lim_{n\rightarrow\infty}g_nx^{(2)}$. For any $F\in F(G)$, denote by $\tilde{F}=\cup_{S\in\mathcal{S}_2}SS^{-1}F$. Then there exists $N\in \N$ such that for any $n>N$ and any $g\in \tilde{F}$, $y_g=x^{(2)}_{gg_n}$. We note that for any $n>N$ and any $g\in F$, if $gg_n$ is contained in some tile $Sc$ of $\mathcal{T}_2$, then $Sc\subset \tilde{F}g_n$. Since $\mathcal{T}_2$ is primely congruent with $\mathcal{T}_1$, for any $g\in G$, if $g$ is contained in some tile $Sc$ of $\mathcal{T}_2$, then $x^{(1)}_g$ is completely determined by $x^{(2)}|_{Sc}$. Hence $x^{(1)}|_{Fg_n}$ is completely determined by $x^{(2)}|_{\tilde{F}g_n}$, which implies that $x^{(1)}|_{Fg_n}$'s are the same up to translations. Thus $\lim_{n\rightarrow\infty}g_nx^{(1)}$ exists and we then let $\pi(y)=\lim_{n\rightarrow\infty}g_nx^{(1)}$.
  A similar argument can show that $\pi$ is continuous. Moreover, for any $x\in X_{\mathcal{T}_2}$ and $g\in G$, $\pi(gx)=g\pi(x)$. Thus $\pi$ is a factor map and $(X_{\mathcal{T}_1},G)$ is a topological factor of $(X_{\mathcal{T}_2},G)$.
\end{proof}
Let $\mathcal{T}$ be a finite tiling of an infinite countable amenable group $G$. Denote by $h(\mathcal{T})=h_{top}(X_{\mathcal{T}},G)$, the topological entropy of the associated subshift $(X_{\mathcal{T}},G)$. The following is Theorem 5.2 of \cite{DHZ} by Downarowicz etc. \cite{DHZ}.
\begin{theorem}\label{tiling1}
Let $G$ be a infinite countable amenable group. Fix a converging to zero sequence $\epsilon_k>0$
and a sequence $K_k$ of finite subsets of $G$. There exists a congruent sequence of finite tilings $\overline{\mathcal{T}}_k$ of $G$ such that the shapes
of $\overline{\mathcal{T}}_k$ are $(K_k,\epsilon_k)$-invariant and $h(\overline{\mathcal{T}}_k)=0$ for each $k$.
\end{theorem}

\begin{remark}\label{uniform}
  In the proof of the above theorem by Downarowicz etc., the congruent sequence of finite tilings $\overline{\mathcal{T}}_k$ can be made
  primely congruent. Hence for any two tiles $Sc_1, Sc_2\in \overline{\mathcal{T}}_{k+1}$ with the same shape $S$, $\overline{\mathcal{T}}_k|_{Sc_1}=(\overline{\mathcal{T}}_k|_{Sc_2})c_2^{-1}c_1$.
\end{remark}

Theorem \ref{tiling1} can be strengthened to make each tiling irreducible.
\begin{theorem}\label{tiling2}
Let $G, (\epsilon_k)$ and $(K_k)$ be as in Theorem \ref{tiling1}. Then there exists a primely congruent sequence of finite irreducible tilings $\mathcal{T}_k$ of $G$
such that the shapes
of $\mathcal{T}_k$ are $(K_k,\epsilon_k)$-invariant and $h(\mathcal{T}_k)=0$ for each $k$.
\end{theorem}
\begin{proof}
  Let $(\overline{\mathcal{T}}_k)$ be the sequence of tilings in Theorem \ref{tiling1} and $(X_{\overline{\mathcal{T}}_k},G)$'s be the associated subshifts. Due to the proof of Theorem 5.2 of \cite{DHZ} by Downarowicz etc.,
  the sequence of tilings $(\overline{\mathcal{T}}_k)$ was chosen to be primely congruent and then there exists for each $k$ a factor map
  $$\pi_{k+1}:(X_{\overline{\mathcal{T}}_{k+1}},G)\rightarrow (X_{\overline{\mathcal{T}}_k},G).$$
  Let $(\tilde{X},G)$ be the inverse limit system of the systems $(X_{\overline{\mathcal{T}}_k},G)$'s and the factor maps $\pi_k$'s.
  Let $\tilde{x}$ be a minimal point of $(\tilde{X},G)$ and let $x_{k}=\tilde{\pi}_k(\tilde{x})$ for each $k$, where $\tilde{\pi}_k$ is the projection map from
  $(\tilde{X},G)$ to $(X_{\overline{\mathcal{T}}_k},G)$.
  Then we obtain a sequence $(x_k)_{k\ge 1}$ such that $x_k$ be a minimal point of $(X_{\overline{\mathcal{T}}_k},G)$ and $\pi_{k+1}(x_{k+1})=x_k$.
  For each $x_k$, we can get a finite tiling $\mathcal{T}_k$ whose shapes are contained in the shapes of $\overline{\mathcal{T}}_k$.
  Due to Lemma \ref{mini}, each $\mathcal{T}_k$ is irreducible.
  Moreover, since $X_{\mathcal{T}_k}\subset X_{\overline{\mathcal{T}}_k}$, $h(\mathcal{T}_k)=0$.
\end{proof}
\begin{theorem}\label{tiling3}
  Let $(\mathcal{T}_n)$ be the primely congruent sequence of finite irreducible tilings of $G$ as in Theorem \ref{tiling2} and let $\{S_{n,1},S_{n,2},\cdots, S_{n,k_n}\}$ be the shapes of $\mathcal{T}_n$ for each $n$.
  Then for each $n$, there exist $l_n\in\N$, $t_n\in G$ and a finite irreducible tiling $\mathcal{T}_{n}'$ such that
  \begin{enumerate}
    \item $e_G\in S_{l_1,1}t_1\subset S_{l_2,1}t_2\subset\cdots$ and $\cup_{n}S_{l_n,1}t_n=G$,
    \item the shapes of $\mathcal{T}_{n}'$ and $\mathcal{T}_{l_n}$ are the same,
    \item $(\mathcal{T}_{n}')$ is also a primely congruent sequence of finite irreducible tilings.
  \end{enumerate}
\end{theorem}
\begin{proof}
  Without loss of generality, we assume that for each $n$, the tile of $\mathcal{T}_n$ that contains $e_G$ (denoted by $T_n$) is of $S_{n,1}$-shape. Moreover, we can replace $S_{n,1}$ by $T_n$ to make $S_{n,1}$ itself be a tile of $\mathcal{T}_n$. Then the shape $S_{n,1}$
  still contains $e_G$.

  Let $G=\{g_1,g_2,\cdots\}$. We first let $l_1=1$ and $t_1=e_G$. By the irreducibility and congruence of $(\mathcal{T}_n)$,
  there exists $l_1'$ sufficiently large to make each shape of $\mathcal{T}_{l_1'}$ sufficiently invariant such that every tile of
  $\mathcal{T}_{l_1'}$ contains at least one tile of $\mathcal{T}_{l_1}$ with $S_{l_1,1}$ shape. Let $l_2$ sufficiently large such that
  $S_{l_2,1}$ is $(\{g_1\}\cup S_{l_1,1}, \frac{1}{\max_{i}|S_{l_1',i}|})$-invariant. Since $e_G\in S_{l_1,1}$,
  \begin{align*}
    \big\{c\in S_{l_2,1}: (\{g_1\}\cup S_{l_1,1})c\nsubseteq S_{l_2,1}\big\}\subseteq B(S_{l_2,1}, \{g_1\}\cup S_{l_1,1}).
  \end{align*}
  Notice that the $\mathcal{T}_{l_2}$-tile $S_{l_2,1}e_G$ contains at least $\lceil\frac{|S_{l_2,1}|}{\max_{i}|S_{l_1',i}|}\rceil$-many $\mathcal{T}_{l_1}$-tiles
  with $S_{l_1,1}$-shape and $|B(S_{l_2,1}, \{g_1\}\cup S_{l_1,1})|<\frac{|S_{l_2,1}|}{\max_{i}|S_{l_1',i}|}$, there must be $c_1\in C(S_{l_1,1})$ such that
  $(\{g_1\}\cup S_{l_1,1})c_1\subseteq S_{l_2,1}$. Hence we have that
  $$e_G\in S_{l_1,1}\subseteq S_{l_2,1}c_1^{-1}\text{ and }g_1\in S_{l_2,1}c_1^{-1}.$$
  Suppose we have taken $l_1, l_2,\ldots, l_n$ and $c_i\in C(S_{l_i,1})$ for $i=1,2,\ldots,n-1$ such that
  $$e_G\in S_{l_1,1}\subseteq S_{l_2,1}c_1^{-1}\subseteq S_{l_3,1}c_2^{-1}c_1^{-1}\subseteq\cdots\subseteq S_{l_n,1}c_{n-1}^{-1}c_{n-2}^{-1}\cdots c_1^{-1}$$
  and
  $$g_i\in S_{l_{i+1},1}c_{i}^{-1}c_{i-1}^{-1}\cdots c_1^{-1}\text { for each }i=1,2,\ldots,n-1.$$
  We then let $l_n'$ sufficiently large to make each shape of $\mathcal{T}_{l_n'}$ sufficiently invariant such that every
  $\mathcal{T}_{l_n'}$-tile contains at least one $\mathcal{T}_{l_n}$-tile with $S_{l_n,1}$ shape. Let $l_{n+1}$ sufficiently large such that
  $S_{l_{n+1},1}$ is $(\{g_nc_1c_2\cdots c_{n-1}\}\cup S_{l_n,1}, \frac{1}{\max_{i}|S_{l_n',i}|})$-invariant. By a similar discussion,
  we can find $c_n\in C(S_{l_n,1})$ such that $(\{g_nc_1c_2\cdots c_{n-1}\}\cup S_{l_n,1})c_n\subseteq S_{l_{n+1},1}$ and hence
  $$S_{l_n,1}c_{n-1}^{-1}c_{n-2}^{-1}\cdots c_1^{-1}\subseteq S_{l_{n+1},1}c_{n}^{-1}c_{n-1}^{-1}\cdots c_1^{-1}\text{ and }g_n\in S_{l_{n+1},1}c_{n}^{-1}c_{n-1}^{-1}\cdots c_1^{-1}.$$
  Let $t_n=c_{n-1}^{-1}c_{n-2}^{-1}\cdots c_1^{-1}$ for each $n>1$, then $l_n$'s and $t_n$'s satisfy (1).

  To obtain $(\mathcal{T}_n')$, we consider the subshifts $(X_{\mathcal{T}_{l_n}}, G)$'s. For each $n$, let $x_n$ be the associated transitive point for the subshift $(X_{\mathcal{T}_{l_n}}, G)$. For each $n>1$, denote by
  $\pi_{n+1}: (X_{\mathcal{T}_{l_{n+1}}}, G)\rightarrow (X_{\mathcal{T}_{l_n}}, G)$ the associated factor map.
  Let $(\tilde{X},G)$ be the inverse limit system of the systems $(X_{\mathcal{T}_{l_n}},G)$'s and the factor maps $\pi_n$'s.
  Let $\tilde {x}\in \tilde{X}$ be the point such that $\tilde{\pi}_n(\tilde{x})=x_n$ for each $n$, where $\tilde{\pi}_n$ is the projection map from $(\tilde{X},G)$ to $(X_{\mathcal{T}_{l_n}},G)$. Let $\tilde{x}'$ be a limit point of the sequence $(c_1c_2\cdots c_n\tilde{x})$.
  Passing to a subsequence if necessary, we may assume that $\tilde{x}'=\lim_{n\rightarrow \infty}c_1c_2\cdots c_n\tilde{x}$.
  Let $x_n'=\pi_n(\tilde{x}')$. For each $x_n'$, we can get a finite irreducible tiling $\mathcal{T}_n'$ whose shapes are the same as
  $\mathcal{T}_{l_n}$'s. Moreover, the tiling sequence $(\mathcal{T}_n')$ is primely congruent.
\end{proof}

\begin{remark}\label{tiling4}
  We should note that for each $\mathcal{T}_n'$, the center of $S_{l_n,1}$ may not contain $e_G$, i.e. $S_{l_n,1}$ itself may not
  be a tile of $\mathcal{T}_n'$. Since for any $k$,
  $(c_nc_{n+1}\cdots c_{n+k}x_n)_{e_G}=(x_n)_{c_nc_{n+1}\cdots c_{n+k}}=S_{l_n,1}$, it is easy to check that
  $S_{l_n,1}c_1c_2\cdots c_{n-1}\in \mathcal{T}_n'$. We can
  replace the shape $S_{l_n,1}$ into $S_{l_n,1}c_{n-1}^{-1}c_{n-2}^{-1}\cdots c_1^{-1}=S_{l_n,1}t_n$, then $e_G\in C(S_{l_n,1}t_n)$.
\end{remark}

\section{Construction}
\subsection{Construction of $X$}
Let $G=\{g_1,g_2,\cdots\}$ and let $P$ be a polyhedron with $0<\dim P<\infty$.
Let $(\delta_n)$ be a sequence decreasing to zero and  let $(P_{\delta_n})$ be an increasing sequence of finite $\delta_n-$dense subset of $P$.
Denote by $\hat{P}=P\cup\{*,\#\}$. For $F\in F(G)$ and $w=(w_g)_{g\in F}\in \hat{P}^F$,
denote by
$$\rho_{*}(w)=\frac{|\{g\in F: w_g=*\}|}{|F|} \text{ and }\rho_{\#}(w)=\frac{|\{g\in F: w_g=\#\}|}{|F|},$$
the density of $*$ and $\#$ in $w$ respectively.
Denote by $$w(F,*)=\{g\in F: w_g=*\} \text{ and }w(F,\#)=\{g\in F: w_g=\#\},$$the subsets of $F$ on which $w$ takes values $*$ and $\#$ respectively.

Let $0<\rho<1$ and let $(\mathcal{T}_n')$ be a primely congruent sequence of finite irreducible tilings of $G$ as in Theorem \ref{tiling3}.
Denote by $\mathcal{S}_n'=\mathcal{S}(\mathcal{T}_n')=\{S_{n,1}',S_{n,2}',\cdots,S_{n,k'_n}'\}$ the shapes of $\mathcal{T}_n'$. By Theorem \ref{tiling3} and Remark \ref{tiling4}, we assume that $e_G\in C(S_{n,1}')$ for each $n$ and
\begin{align}\label{condition1}
  e_G\in S_{1,1}'\subset S_{2,1}'\subset\cdots \text{ and }\cup_{n\ge1}S_{n,1}'=G.
\end{align}

{\bf Step 1.} Take $\mathcal{T}_1=\mathcal{T}_1'$. Denote by $\mathcal{S}_1=\mathcal{S}(\mathcal{T}_1)=\{S_{1,i}: i=1,2,\ldots,k_1\}$ with $k_1=k_1'$
and $S_{1,i}=S_{1,i}'$ for each $i=1,2,\ldots,k_1$.
For each shape $S_{1,i}, i=1,2,\ldots,k_1$, let $A_i$ be a subset of $S_{1,i}$ with
$$\rho<\frac{|A_i|}{|S_{1,i}|}\le\rho+\frac{1}{|S_{1,i}|}.$$

Define $w_{1,i}\in \hat{P}^{S_{1,i}}$ by
\begin{align*}
  (w_{1,i})_g=\begin{cases}
               &*, \text{ if }g\in A_i,\\
               &\#, \text{ otherwise.}
              \end{cases}
\end{align*}
Then define $w_1\in \hat{P}^G$ by
$$w_1|_{S_{1,i}c}=w_{1,i}, \text{ for any tile }S_{1,i}c\in \mathcal{T}_1.$$

{\bf Step 2.} Take $l_1$ large enough such that any tile of $\mathcal{T}_{l_1}'$ with shape $S'_{l_1,i}$ contains more than $|P_{\delta_1}|^{|w_1(S_{1,1}, *)|}$-
many elements in $\mathcal{T}_1$ with $S_{1,1}$-shape. Take $R_1\subset C(S_{l_1,1}')$ and $h_1\in C(S_{l_1,1}')$ such that
\begin{enumerate}
  \item[1] $e_G\in R_1$,
  \item[2] $|R_1|=|P_{\delta_1}|^{|w_1(S_{1,1}, *)|} \text{ and }S_{1,1}R_1\subset S_{l_1,1}'$,
  \item[3] $h_1\notin R_1 \text{ and }S_{1,1}h_1\subset S_{l_1,1}'$.
\end{enumerate}

Define a word $w_1'\in \hat{P}^{S_{l_1,1}'}$ by modifying $w_1|_{S_{l_1,1}'}$ on $w_1(S_{1,1}R_1,*)$ such that: 
\begin{enumerate}

\item if $g\in S_{l_1,1}'\setminus w_1(S_{1,1}R_1,*)$, then $(w_1')_g=(w_1)_g$, if $g\in w_1(S_{1,1}R_1,*)$, then $(w_1')_g\in P_{\delta_1}$;
\item for $r\in R_n$, $w_1'|_{w_1(S_{1,1},*)r}$'s are mutually different, i.e. $\{w_1'|_{w_1(S_{1,1},*)r}\}_{r\in R_1}=\{P_{\delta_1}^{w_1(S_{1,1},*)}\}.$
\end{enumerate}

In the following we will take $m_1$ large enough and then take $\mathcal{T}_2=\mathcal{T}_{m_1}'$. Denote by
$\mathcal{S}_2=\mathcal{S}(\mathcal{T}_2)=\{S_{2,i}: i=1,2,\ldots,k_2\}$ with $k_2=k_{m_1}'$ and $S_{2,i}=S_{m_1,i}'$ for each $i=1,2,\cdots, k_2$.
The requirements of $\mathcal{T}_2$ is the following:
\begin{enumerate}
  \item $S_{2,1}$ contains $g_1h_1$;
  \item any tile $S_{2,i}c$ of $\mathcal{T}_2$ contains some $S_{l_1,1}'c'\in \mathcal{T}_{l_1}'$, moreover, we assume for $S_{2,i}c$'s, $S_{l_1,1}'c'$'s are
  chosen such that $S_{l_1,1}'c'c^{-1}$'s are the same due to the requirement in Remark \ref{uniform};
  \item $$\frac{|w_1(S_{2,i}c\setminus S_{l_1,1}'c',*)|}{|S_{2,i}c|}>\rho.$$
\end{enumerate}
By condition \eqref{condition1}, (1) can be satisfied when $m_1$ is sufficiently large.
(2) can be satisfied by the irreducibility of $\mathcal{T}_{l_1}'$ whenever $S_{2,i}$'s are sufficiently invariant.
Since any tile $S_{2,i}c$ of $\mathcal{T}_2$ is the disjoint union of some tiles of $\mathcal{T}_1$ and on each tile $T$ of $\mathcal{T}_1$,
$\rho_*(w_1|_{T})>\rho$, when the size of $S_{l_1,1}'c'$ is negligible compared with $S_{2,i}c$, (3) can be satisfied.

For each $S_{2,i}c\in\mathcal{T}_2$, define $v_{1,i}\in \hat{P}^{S_{2,i}c}$ (for different $S_{2,i}c$'s, $v_{1,i}$'s are the same up to translations) such that
\begin{enumerate}
  \item $v_{1,i}|_{S_{l_1,1}'c'}=w_1'$;
  \item on $w_1|_{S_{2,i}c\setminus S_{l_1,1}'c'}$, change the $*$'s into $\#$'s as many as possible such that
  \begin{enumerate}
    \item $\rho_*(v_{1,i})>\rho$ always holds;
    \item for each $S_{1,j}\hat{c}\in \mathcal{T}_1$ with $S_{1,j}\hat{c}\subset S_{2,i}c\setminus S_{l_1,1}'c'$, $$\rho_*(v_{1,i}|_{S_{1,j}\hat{c}})>\rho-\frac{1}{|S_{1,j}|};$$
    \item we first change the $*$'s into $\#$'s for tiles with $S_{1,k_1}$-shape, then tiles with $S_{1,k_1-1}$-shape and so on.
  \end{enumerate}
\end{enumerate}

Define $w_2\in \hat{P}^G$ such that $w_2|_{S_{2,i}c}=v_{1,i}$ for each $S_{2,i}c\in\mathcal{T}_2$.

Suppose that we have finished step $n$ for $n\ge 2$. Using a similar way as step 2, we will obtain $\mathcal{T}_{n+1}$ and $w_{n+1}\in \hat{P}^G$ from $\mathcal{T}_n$ at step $n+1$.

{\bf Step $n+1$.} Take $l_n$ large enough such that any tile of $\mathcal{T}_{l_n}'$ with shape $S'_{l_n,i}$ contains more than $|P_{\delta_n}|^{|w_n(S_{n,1}, *)|}$-
many elements in $\mathcal{T}_n$ with $S_{n,1}$-shape. Take $R_n\subset C(S_{l_n,1}')$ and $h_n\in C(S_{l_n,1}')$ such that
\begin{enumerate}
  \item $e_G\in R_n$,
  \item $|R_n|=|P_{\delta_n}|^{|w_n(S_{n,1}, *)|} \text{ and }S_{n,1}R_n\subset S_{l_n,1}'$,
  \item $h_n\notin R_n \text{ and }S_{n,1}h_n\subset S_{l_n,1}'$.
\end{enumerate}

Similar with step 2, we modify the *'s of $w_n|_{S_{l_n,1}'}$ on $S_{n,1}R_n$ without changing other places to obtain $w_n'\in \hat{P}^{S_{l_n,1}'}$ such that $$\{w_n'|_{w_n(S_{n,1},*)r}\}_{r\in R_n}=\{P_{\delta_n}^{w_n(S_{n,1},*)}\}.$$
We note that $w'_n|_{S_{n,1}h_n}=w_n|_{S_{n,1}}=v_{n-1,1}|_{S_{n,1}}$ and no $*$ occurs in $w'_n|_{S_{n,1}}$.

In the following we will take $m_n$ large enough and then take $\mathcal{T}_{n+1}=\mathcal{T}_{m_n}'$. Denote by
$\mathcal{S}_{n+1}=\mathcal{S}(\mathcal{T}_{n+1})=\{S_{n+1,i}: i=1,2,\ldots,k_{n+1}\}$ with $k_{n+1}=k_{m_n}'$ and $S_{n+1,i}=S_{m_n,i}'$ for each $i=1,2,\cdots, k_{n+1}$.
The requirements of $\mathcal{T}_{n+1}$ is the following:
\begin{enumerate}
  \item $S_{n+1,1}$ contains $g_nh_1h_2\cdots h_n$;
  \item any tile $S_{n+1,i}c$ of $\mathcal{T}_{n+1}$ contains some $S_{l_n,1}'c'\in \mathcal{T}_{l_n}'$, moreover, we assume for $S_{n+1,i}c$'s, $S_{l_n,1}'c'$'s are
  chosen such that $S_{l_n,1}'c'c^{-1}$'s are the same due to the requirement in Remark \ref{uniform};
  \item $$\frac{|w_1(S_{n+1,i}c\setminus S_{l_n,1}'c',*)|}{|S_{n+1,i}c|}>\rho.$$
\end{enumerate}
By condition \eqref{condition1}, (1) can be satisfied when $m_n$ is sufficiently large.
(2) can be satisfied by the irreducibility of $\mathcal{T}_{l_n}'$ whenever $S_{n+1,i}$'s are sufficiently invariant.
Since any tile $S_{n+1,i}c$ of $\mathcal{T}_{n+1}$ is the disjoint union of some tiles of $\mathcal{T}_n$ and on each tile $T$ of $\mathcal{T}_n$,
$\rho_*(w_n|_{T})>\rho$, when the size of $S_{l_n,1}'c'$ is negligible compared with $S_{n+1,i}c$, (3) can be satisfied.

For each $S_{n+1,i}c\in\mathcal{T}_{n+1}$, define $v_{n,i}\in \hat{P}^{S_{n+1,i}c}$ (for different $S_{n+1,i}c$'s, $v_{n,i}$'s are the same up to translations) such that
\begin{enumerate}
  \item $v_{n,i}|_{S_{l_n,1}'c'}=w_n'$;
  \item on $w_n|_{S_{n+1,i}c\setminus S_{l_n,1}'c'}$, change the $*$'s into $\#$'s as many as possible such that
  \begin{enumerate}
    \item $\rho_*(v_{n,i})>\rho$ always holds;
    \item for each $S_{n,j}\hat{c}\in \mathcal{T}_n$ with $S_{n,j}\hat{c}\subset S_{n+1,i}c\setminus S_{l_n,1}'c'$, $$\rho_*(v_{n,i}|_{S_{n,j}\hat{c}})>\rho-\frac{1}{|S_{n,j}|};$$
    \item we first change the $*$'s into $\#$'s for tiles with $S_{n,k_n}$-shape, then tiles with $S_{n,k_n-1}$-shape and so on.
  \end{enumerate}
\end{enumerate}
Hence
\begin{align}\label{vn1}
  v_{n,1}|_{S_{n,1}h_n}=w'_n|_{S_{n,1}h_n}=v_{n-1,1}|_{S_{n,1}h_n}=v_{n-1,1}|_{S_{n,1}}.
\end{align} Moreover,
by (a) of (2), after $v_{n,i}$ is defined, we have that
\begin{align}\label{rho}
  \rho<\rho_*(v_{n,i})\le\rho+\frac{1}{|S_{n+1,i}|}.
\end{align}
Define $w_{n+1}\in \hat{P}^G$ such that $w_{n+1}|_{S_{n+1,i}c}=v_{n,i}$ for each $S_{n+1,i}c\in\mathcal{T}_{n+1}$.

Since $S_{n,1}\subset S_{n+1,1}$, we have that $w_{n+1}|_{S_{n,1}}=v_{n,1}|_{S_{n,1}}=w'_n|_{S_{n,1}}$. Because $w'_n|_{S_{n,1}}$
contains no $*$'s, in the proceeding steps, for $k>1$, the restriction of $w_{n+k}$'s to $S_{n,1}$ is preserved and equal to $w'_n|_{S_{n,1}}$. Noticing that $S_{n,1}\nearrow G$, the limit of the sequence $(w_n)_{n\in\N}$ exists. We then let $w=\lim_{n\rightarrow \infty}w_n$.

Note that for each $g\in G$, $w_g$ cannot be $*$. Replace the $\#$'s of $w$ into some $p\in P$ to obtain $x\in P^G$. Then our subshift is defined by $X=\overline{\cup_{g\in G}gx}$. 

To summarize, we list some properties of the words appeared in our constructions: 
\begin{enumerate}
  \item[1.] $v_{n,i}$ is defined on tiles in $\mathcal{T}_{n+1}$ of $S_{n+1,i}$-shape and the restrictions of $v_{n,i}$ to these tiles are the same (up to translations).
  \item[2.] $w_{n+1}$ is defined on the whole $G$ and $w_{n+1}|_{S_{n+1,i}c}=v_{n,i}$ for each $\mathcal{T}_{n+1}$ tile $S_{n+1,i}c$.
  \item[3.] $w_n'$ is defined on $S_{l_n,1}'$ and $S_{n,1}\subset S_{l_n,1}'\subset S_{n+1,1}$. 
  \item[4.] For each $k>1$, $w_{n+k}|_{S_{n,1}}=w_{n+1}|_{S_{n,1}}=v_{n,1}|_{S_{n,1}}=w'_n|_{S_{n,1}}$ and no * occurs in this word.
  \item[5.] $v_{n,1}|_{S_{n,1}h_n}=w_n'|_{S_{n,1}h_n}=w_n|_{S_{n,1}}=v_{n-1,1}|_{S_{n,1}}$.
\end{enumerate}

\subsection{Lower bound for ${\rm mdim}(X,G)$}
For any $n\in\N$, due to the definition of $v_{n,1}$ in step $n+1$, $S_{n,1}h_n\subset S_{n+1,1}$. Hence we can obtain a sequence of increasing finite subset of $G$:
$$S_{1,1}\subset S_{2,1}h_1^{-1}\subset\ldots\subset S_{n+1,1}h_n^{-1}h_{n-1}^{-1}\cdots h_1^{-1}\subset \ldots.$$ Since $g_nh_1h_2\cdots h_n\in S_{n+1,1}$ for each $n$,
we can deduce that $$\{g_1,g_2,\cdots,g_n\}\subset S_{n+1,1}h_n^{-1}h_{n-1}^{-1}\cdots h_1^{-1}.$$ Hence $$\cup_{n=1}^\infty S_{n+1,1}h_n^{-1}h_{n-1}^{-1}\cdots h_1^{-1}=G.$$
We note that $\{S_{n+1,1}h_n^{-1}h_{n-1}^{-1}\cdots h_1^{-1}\}$ forms a F{\o}lner sequence of $G$.

To give the lower bound for ${\rm mdim} (X,G)$, we first let $$J_n=\{g\in S_{n+1,1}:(v_{n,1})_g=*\}h_n^{-1}h_{n-1}^{-1}\cdots h_1^{-1}.$$
By \eqref{vn1},
\begin{align}\label{vn2}
  \{g\in S_{n,1}: (v_{n-1,1})_g=*\}h_n&=\{g\in S_{n,1}h_n: (v_{n-1,1})_g=*\}\ \nonumber \\
  &\subset \{g\in S_{n+1,1}: (v_{n,1})_g=*\}.
\end{align}
Hence $J_{n-1}\subset J_n$
and we then let $$J=\cup_{n=1}^\infty J_n=\lim_{n\rightarrow \infty}J_n.$$
Since $$\frac{|J\cap S_{n+1,1}h_n^{-1}h_{n-1}^{-1}\cdots h_1^{-1}|}{|S_{n+1,1}h_n^{-1}h_{n-1}^{-1}\cdots h_1^{-1}|}\ge\frac{|J_n|}{|S_{n+1,1}|}=\rho_*(v_{n,1})>\rho,$$
we have that $\delta(J)\ge\rho$.

Now we will find an $x'\in X$ such that for any $z\in P^G$,
$$\pi_{G\setminus J}(z)=\pi_{G\setminus J}(x')\Rightarrow z\in X.$$

$\forall u\in P^J$, for any $t>0$, we define $u_t\in (P\cup \{\#\})^{S_{t+1,1}}$ by the following:
\begin{align*}
  (u_t)_g=\begin{cases}
    (u)_{gh_t^{-1}h_{t-1}^{-1}\cdots h_1^{-1}}, &\text { if }g\in J_th_1h_2\cdots h_t; \\
    (v_{t,1})_g, &\text{ if }g\in S_{t+1,1}\setminus J_th_1h_2\cdots h_t,
  \end{cases}
\end{align*}
i.e. $u_{t,n}|_{v_{t,1}(S_{t+1,1},*)}=u_{n}'|_{v_{t,1}(S_{t+1,1},*)}$ and $u_{t,n}|_{S_{t+1,1}\setminus v_{t,1}(S_{t+1,1},*)}=v_{t,1}|_{S_{t+1,1}\setminus v_{t,1}(S_{t+1,1},*)}$.

We can find $u_n'\in (P_{\delta_n})^J$ for each $n>0$ such that $\lim_{n\rightarrow\infty}(u_n')_g=(u)_g$ for any $g\in J$.
When $n>t$, we define $u_{t,n}\in(P\cup \{\#\})^{S_{t+1,1}}$ such that
\begin{align*}
  (u_{t,n})_g=\begin{cases}
    (u_n')_{gh_t^{-1}h_{t-1}^{-1}\cdots h_1^{-1}}, &\text { if }g\in J_th_1h_2\cdots h_t; \\
    (v_{t,1})_g, &\text{ if }g\in S_{t+1,1}\setminus J_th_1h_2\cdots h_t.
  \end{cases}
\end{align*}
Hence $u_t=\lim_{n\rightarrow\infty}u_{t,n}$. 

We note that $S_{t+1,1}h_{t+1}h_{t+2}\cdots h_{n-1}\subset S_{t+2,1}h_{t+2}h_{t+3}\cdots h_{n-1}\subset \cdots\subset S_{n,1}$.
By \eqref{vn1} again
\begin{align*}
  v_{t,1}|_{S_{t+1,1}}&=v_{t,1}|_{S_{t+1,1}h_{t+1}h_{t+2}\cdots h_{n-1}}\\
  &=v_{t+2,1}|_{S_{t+1,1}h_{t+1}h_{t+2}\cdots h_{n-1}}\\
  &=\cdots\\
  &=v_{n-1,1}|_{S_{t+1,1}h_{t+1}h_{t+2}\cdots h_{n-1}}.
\end{align*}
By \eqref{vn2},
\begin{align*}
  \{g\in S_{t+1,1}: (v_{t,1})_g=*\}h_{t+1}h_{t+2}\cdots h_{n-1}&\subset \{g\in S_{n,1}: (v_{n-1,1})_g=*\}\\
  &=w_n(S_{n,1},*).
\end{align*}
Due to the construction of $w_n'$, any combination of points in $P_{\delta_n}$ can be found in $\{w_n'|_{w_n(S_{n,1},*)r}:r\in R_n\}$. 
Hence any combination of points in $P_{\delta_n}$ can also be found in $\{w_n'|_{w_n(S_{t+1,1}h_{t+1}h_{t+2}\cdots h_{n-1},*)r}:r\in R_n\}$.
Thus we can find some $r_{t,n}\in R_n$ such that
\begin{align*}
  u_{t,n}&=w_n'|_{S_{t+1,1}h_{t+1}h_{t+2}\cdots h_{n-1}r_{t,n}}\\
         &=w_{n+1}|_{S_{t+1,1}h_{t+1}h_{t+2}\cdots h_{n-1}r_{t,n}}\\
         &=w|_{S_{t+1,1}h_{t+1}h_{t+2}\cdots h_{n-1}r_{t,n}}\\
         &=(h_{t+1}h_{t+2}\cdots h_{n-1}r_{t,n}w)|_{S_{t+1,1}}.
\end{align*}
Changing all $\#$'s into $p$ and taking a limit point of the sequence $(h_{t+1}h_{t+2}\cdots h_{n-1}r_{t,n}w)$ as
$n$ goes to infinity, we can find $x_{u_t}\in X$ such that $x_{u_t}|_{S_{t+1,1}}=u_t$. Letting $t\rightarrow\infty$, we then define $x'\in X$
to be a limit point of $(h_1h_2\cdots h_tx_{u_t})$. It follows that $x'|_J=u$ since
$(h_1h_2\cdots h_tx_{u_t})|_{J_t}=x_{u_t}|_{J_th_1h_2\cdots h_t}=u|_{J_t}$.

Now let $z\in P^G$ such that $\pi_{G\setminus J}(z)=\pi_{G\setminus J}(x')$. Let $\theta=z|_J$ and replace $u$ of the above process into $\theta$.
It is easy to show that $z\in X$. Applying Proposition \ref{lowerbound}, ${\rm mdim}(X,G)\ge \rho \dim P$.

\subsection{Upper bound for ${\rm mdim} (X,G)$}

For any $n\in\N$ and $\epsilon>0$, whenever $F_n\in F(G)$ is sufficiently invariant, when we see the restriction of $\mathcal{T}_n$ on $F_n$,
the union of $\mathcal{T}_n$-tiles which are contained in $F_n$ has a portion bigger than $1-\epsilon$. The same thing holds for any translation of $\mathcal{T}_n$.

Due to the patterns of the restrictions $\mathcal{T}_ng|_{F_n}$, $g\in G$,
we can partition $G$ into $L$-many classes $Q_1,Q_2,\ldots,Q_L$ and in each class $Q_j$,
the restrictions $\mathcal{T}_ng|_{F_n}$, $g\in Q_j$, are of the same pattern, i.e. for any $g_1, g_2\in Q_j, \mathcal{T}_ng_1|_{F_n}=\mathcal{T}_ng_2|_{F_n}.$
Here $L$ is the total number of such different patterns, which is finite and depends on $F_n$ and the shapes of $\mathcal{T}_n$.

For any tile $S_{n,i}c\in\mathcal{T}_n$, let us compare $x|_{S_{n,i}c}$ and $v_{n-1,i}$.
Due to the construction of $x$, for any $h\in S_{n,i}c$, $x_h=(v_{n-1,i})_h$ whenever $(v_{n-1,i})_h\neq *$.
Let $1\le j\le L$ be fixed. For $g\in Q_j$, $(g^{-1}x)|_{F_n}$'s coincide on the following set
$$\{h\in S_{n,i}cg\cap F_n: S_{n,i}c\in \mathcal{T}_n
\text{ and } (v_{n-1,i})_{hg^{-1}}\neq *\}.$$
Denoted this set by $C$. Then $$\dim \overline{\{(g^{-1}x)|_{F_n}: g\in Q_j\}}\le |F_n\setminus C|.$$
Noticing that
\begin{align*}
F_n\setminus C&=\{h\in S_{n,i}cg\cap F_n: S_{n,i}c\in \mathcal{T}_n
\text{ and } (v_{n-1,i})_{hg^{-1}}= *\}\\
&\subset\{h\in S_{n,i}cg\subset F_n: S_{n,i}c\in \mathcal{T}_n
\text{ and } (v_{n-1,i})_{hg^{-1}}= *\} \\
&\qquad \cup (F_n\setminus \{h\in S_{n,i}cg\subset F_n: S_{n,i}c\in \mathcal{T}_n\}),
\end{align*}
we have the following estimate:
\begin{align*}
  |F_n\setminus C|&\le \sum_{S_{n,i}cg\subset F_n,S_{n,i}c\in \mathcal{T}_n}\rho_* (v_{n-1,i})|S_{n,i}|+\epsilon|F_n|\\
  &\le (\rho+\frac{1}{\min_{1\le i\le k_n}|S_{n,i}|})|F_n|+\epsilon|F_n|\quad (\text{by \eqref{rho}}).
\end{align*}
Since $$\pi_{F_n}(X)=\overline{\{(g^{-1}x)|_{F_n}: g\in G\}}=\cup_{j=1}^L\overline{\{(g^{-1}x)|_{F_n}: g\in Q_j\}},$$
we have that
$$\dim \pi_{F_n}(X)\le \max_{1\le j\le L}\dim \overline{\{(g^{-1}x)|_{F_n}: g\in Q_j\}}.$$
Thus
$$\frac{\dim \pi_{F_n}(X)}{|F_n|}\le \rho+\frac{1}{\min_{1\le i\le k_n}|S_{n,i}|}+\epsilon.$$
Hence by Proposition \ref{upperbound},
$${\rm mdim} (X,G)\le \liminf_{n\rightarrow\infty}\frac{\dim(\pi_{F_n}(X))}{|F_n|}\le \rho.$$



\subsection{Minimality of $(X,G)$}

To show $(X,G)$ is minimal, it suffices to show that for any $n$, the set $\{g\in G: (gx)|_{S_{n,1}}=x|_{S_{n,1}}\}$ is syndetic.

Noticing that after step $n+1$, in each tile $S_{n+1,1}c\in\mathcal{T}_{n+1}$, $w|_{S_{n,1}c}=w_{n+1}|_{S_{n,1}c}$, we have that
$\{g\in G: (gw)|_{S_{n,1}}=w|_{S_{n,1}}\}\supset C(S_{n+1,1})$.
Since $\mathcal{T}_{n+1}$ is irreducible, by Lemma \ref{irreducible}, $C(S_{n+1,1})$ is syndetic.
Hence $$\{g\in G: (gx)|_{S_{n,1}}=x|_{S_{n,1}}\}\supset \{g\in G: (gw)|_{S_{n,1}}=w|_{S_{n,1}}\}$$ is also syndetic, which shows that $(X,G)$ is minimal.

Finally, we ask the following question:
\begin{question}
Does there exist a minimal subshift of $P^G$ whose mean topological dimension equals $\dim P$?
\end{question}

{\bf Acknowledgements}
The author would like to thank Prof. Hanfeng Li for his guidance and valuable suggestion. This work was done when the author was a visiting scholar at SUNY at Buffalo and he would like to thank CSC and SUNY at Buffalo for the support.
This work was also partially supported by NNSF of China(Grant No. 10901080, 11401220 and 11431012).

\end{document}